\documentclass[a4paper,11pt]{amsart}

\makeatletter
\newcommand{\addresseshere}{%
  \enddoc@text\let\enddoc@text\relax
}
\makeatother

\author[A.~Bazzanella]{Alice Bazzanella}
\author[C.~Sanna]{Carlo Sanna$^\dagger$}
\thanks{$\dagger\,$C.~Sanna is a member of GNSAGA of INdAM and of CrypTO, the group of Cryptography and Number~Theory of Politecnico di Torino}
\address{\parbox{\linewidth}{
Department of Mathematical Sciences, Politecnico di Torino\\
Corso Duca degli Abruzzi 24, 10129 Torino, Italy\\[-8pt]}}
\email{carlo.sanna@polito.it}


\title[Another factor of integer polynomials with minimal integrals]{Another factor of integer polynomials with\\ minimal integrals}

\usepackage{amsmath}
\usepackage{amssymb}
\usepackage{amsthm}
\usepackage{soul}
\usepackage[left=1.15in, right=1.15in, top=.72in, bottom=.72in]{geometry}
\usepackage[colorlinks=true]{hyperref}
\usepackage{cleveref}
\usepackage{enumitem}
\setlist[enumerate]{label=(\roman*),labelindent=1em,itemsep=0.5em,topsep=0.5em}

\usepackage{listings}
\usepackage{xcolor}
\DeclareMathOperator{\lcm}{lcm}

\newtheorem{theorem}{Theorem}[section]

\newtheorem{lemma}[theorem]{Lemma}
\theoremstyle{remark}

\begin{document}

\begin{abstract}
	Let $N$ be a positive integer and let $S_N$ be the set of polynomials with integer coefficients, degree less than $N$, and minimal positive integral over $[0,1]$.
    D.~Bazzanella initiated the study of $S_N$ because of its relation to the distribution of prime numbers.
    Indeed, it is possible to prove that $\sum_{p^m \leq N} \log p = -\log \int_0^1 P(x) \mathrm{d} x$ for every $P \in S_N$, where the sum runs over prime numbers $p$ and positive integers $m$ such that $p^m \leq N$.
    For each real number $t$, let $\lfloor t \rfloor$ denote the maximal integer not exceeding $t$.
    The main result of this paper states that there exist infinitely many polynomials $P \in S_N$ such that $\big(x^3(1 - x)^2\big)^{\lfloor N / 6 \rfloor}$ divides $P(x)$ in $\mathbb{Z}[x]$.
    This improves upon a similar result of Sanna, who proved the same claim but with the lower-degree polynomial $\big(x(1-x)\big)^{\lfloor N / 3 \rfloor}$ in place of $\big(x^3(1 - x)^2\big)^{\lfloor N / 6 \rfloor}$.
\end{abstract}

\maketitle

\section{Introduction}

In 1896 Hadamard and de~la~Vall\'ee Poussin independently proved the Prime Number Theorem, which is equivalent to the assertion that $\psi(N) \sim N$ as $N \to \infty$.
Here $\psi$ is the second Chebyshev function, which is defined by
\begin{equation*}
    \psi(N) := \sum_{p^m \,\leq\, N} \log p
\end{equation*}
for every positive integer $N$, where the sum runs over prime numbers $p$ and positive integers $m$ such that $p^m \leq N$.

In 1936 Gelfond and Shnirelman proposed an elementary method to prove inequalities of the form $\psi(N) > C N$ for all integers $N > N_0$, where $C$ and $N_0$ are effectively computable positive constants (see Gelfond's editorial remarks in the 1944 edition of Chebyshev's Collected Works \cite[p.\ 287--288]{MR12075}).
In 1982 Nair~\cite{MR657495,MR643279} rediscovered and further developed this method.

The Gelfond--Shnirelman--Nair method works as follows.
Let $P_N(x) = \sum_{n = 0}^{N - 1} c_n x^n$ (\mbox{$c_n \in \mathbb{Z}$}) be a polynomial with integer coefficients and degree less than $N$.
Moreover, set $I(P) := \int_0^1 P(x)\mathrm{d}x$ for every polynomial $P \in \mathbb{Z}[x]$.
Then
\begin{equation*}
    I(P_N) = \sum_{n \,=\, 0}^{N - 1} \frac{c_n}{n + 1}
\end{equation*}
is a rational number whose denominator divides the lowest common multiple
\begin{equation*}
    \ell_N := \lcm\{1,2,\dots, N\} .
\end{equation*}
Suppose that $I(P_N) \neq 0$.
Hence $\ell_N |I(P_N)| \geq 1$.
Furthermore, since $\psi(N) = \log \ell_N$, the inequality
\begin{equation*}
    \psi(N) \geq \log\!\left(\frac1{|I(P_N)|}\right)
\end{equation*}
follows.
In turn $|I(P_N)| \leq \|P_N\|$, where
\begin{equation*}
    \|P\| := \max_{x \,\in\, [0,1]} |P(x)|
\end{equation*}
is the uniform norm of a polynomial $P \in \mathbb{Z}[x]$.
Therefore
\begin{equation}\label{equ:psi-lower-bound-norm}
    \psi(N) \geq \log\!\left(\frac1{\|P_N\|}\right) .
\end{equation}
At this point, it is a matter of choosing a polynomial $P_N$ that has a sufficiently small (positive) norm.
For every real number $t$, let $\lfloor t \rfloor$ denote the maximal integer that does not exceed $t$.
Plugging the polynomial
\begin{equation*}
    P_N(x) = \big(x(1-x)\big)^{\lfloor (N-1) / 2 \rfloor}
\end{equation*}
into \eqref{equ:psi-lower-bound-norm} yields the inequality
\begin{equation*}
    \psi(N) \geq (\log 2) (N - 2) > 0.6931 \cdot (N - 2) .
\end{equation*}
Other polynomials produce stronger inequalities.

This motivated the so-called Integer Chebyshev Problem~\cite{MR1333305}, that is, the study of the quantities
\begin{align*}
    m_N &:= \min\!\big\{\|P\| : P \in \mathbb{Z}[x],\, P \neq 0,\, \deg(P) < N \big\} , \\
    C_N &:= \frac1{N} \log\!\left(\frac1{m_N}\right) ,
\end{align*}
and the set of polynomials
\begin{equation*}
    T_N := \big\{P \in \mathbb{Z}[x] : \|P\| = m_N,\, \deg(P) < N \big\} .
\end{equation*}
In particular, in 1988 Aparicio~\cite{MR968933} proved the following theorem regarding the factors of polynomials in $T_N$.

\begin{theorem}\label{thm:aparicio}
    There exist three constants
    \begin{equation*}
        \lambda_1 \in [0.1456, 0.1495], \quad
        \lambda_2 \in [0.0166, 0.0187], \quad
        \lambda_3 \in [0.0037, 0.0053]
    \end{equation*}
    such that, for every sufficiently large integer $N$ and each $P \in T_N$, the polynomial
    \begin{equation*}
        \big(x(1 - x)\big)^{\lfloor \lambda_1 N \rfloor} (2x - 1)^{\lfloor \lambda_2 N \rfloor} (5x^2 - 5x + 1)^{\lfloor \lambda_3 N \rfloor}
    \end{equation*}
    divides $P(x)$ in $\mathbb{Z}[x]$.
\end{theorem}

In 1994 Montgomery~\cite[Chapter 10]{MR1297543} proved that $C_N$ converges to a limit $C$ as $N \to \infty$; and in 2005 Pritsker~\cite[Theorem 3.1]{MR2177184} showed that $C \in [0.86441, 0.85992]$.
Consequently, the Gelfond--Shnirelman--Nair method cannot lead to an inequality that is stronger than
\begin{equation}\label{equ:psi-0.85992N}
    \psi(N) > 0.85992 \cdot N
\end{equation}
for all large integers $N$.
Unfortunately, inequality \eqref{equ:psi-0.85992N} is quite far from the asymptotic lower bound of the Prime Number Theorem.

To deal with this problem, in 2013 D.~Bazzanella~\cite{MR3122290} (see also~\cite{MR3498948,MR3675407}) suggested to study the polynomials $P_N$ for which $|I(P_N)|$ is nonzero and minimal or, without loss of generality, for which $I(P_N)$ is positive and minimal.
Thus, in light of the inequality $\ell_n |I(P_N)| \geq 1$, he defined the set
\begin{equation*}
    S_N := \big\{P \in \mathbb{Z}[x] : \deg(P) < N, \, I(P) = 1 / \ell_N \big\}
\end{equation*}
and initiated the study of its properties.

In particular, D.~Bazzanella (see~\cite[Theorem~1]{MR3122290} and the comments after~\cite[Theorem~1.3]{MR3682482}) gave a result that implies the following theorem, which is somehow similar to Theorem~\ref{thm:aparicio}.

\begin{theorem}\label{thm:bazzanella}
    For every positive integer $N$, there exist infinitely many polynomials \mbox{$P \in S_N$} such that $x^{\lfloor N / 2 \rfloor}$ divides $P(x)$ in $\mathbb{Z}[x]$.
    Moreover, the exponent $x^{\lfloor N / 2 \rfloor}$ cannot be improved, that is, there exist infinitely many positive integers $N$ such that $x^{\lfloor N / 2 \rfloor + 1}$ divides no polynomial in $S_N$.
    The same results hold if $x^{\lfloor N / 2 \rfloor}$ is replaced by $(1 - x)^{\lfloor N / 2 \rfloor}$.
\end{theorem}

Later, in 2017, Sanna~\cite{MR3682482} proved the following theorem.

\begin{theorem}\label{thm:sanna}
    For every positive integer $N$, there exist infinitely many polynomials \mbox{$P \in S_N$} such that
    \begin{equation*}
        \big(x(1 - x)\big)^{\lfloor N / 3 \rfloor}
    \end{equation*}
    divides $P(x)$ in $\mathbb{Z}[x]$.
    Moreover, the exponent $\lfloor N / 3 \rfloor$ cannot be improved, that is, there exist infinitely many positive integers $N$ such that
    \begin{equation*}
        \big(x(1 - x)\big)^{\lfloor N / 3 \rfloor + 1}
    \end{equation*}
    divides no polynomial in $S_N$.
\end{theorem}

Let $\{Q_N(x)\}$ be a sequence of ``explicit'' polynomials with the property that, for every positive integer $N$, there exists a polynomial $P \in S_N$ such that $Q_N(x)$ divides $P(x)$ in $\mathbb{Z}[x]$.
In~light of Theorems~\ref{thm:bazzanella} and~\ref{thm:sanna}, three examples of such sequences are $\big\{x^{\lfloor N / 2 \rfloor}\big\}$, $\big\{(1-x)^{\lfloor N / 2 \rfloor}\big\}$, and $\big\{\big(x(1-x)\big)^{\lfloor N / 3 \rfloor}\big\}$.
Sanna~\cite[Question]{MR3682482} asked how large the limit
\begin{equation*}
    \delta_Q := \liminf_{N \,\to\, \infty} \frac{\deg(Q_N)}{N}
\end{equation*}
can be; and if $\delta_Q$ can be arbitrarily close to $1$, or even equal to $1$.
Note that the sequences of Theorems~\ref{thm:bazzanella} and~\ref{thm:sanna} give $\delta_Q = 1/2$ and $\delta_Q = 2/3$, respectively.

The main result of this paper is the following.

\begin{theorem}\label{thm:main}
    For every positive integer $N$, there exist infinitely many polynomials \mbox{$P \in S_N$} such that
    \begin{equation*}
        \big(x^3(1 - x)^2\big)^{\lfloor N / 6 \rfloor}
    \end{equation*}
    divides $P(x)$ in $\mathbb{Z}[x]$.
\end{theorem}

Note that the sequence of Theorem~\ref{thm:main} gives $\delta_Q = 5/6$.
The next result shows that Theorem~\ref{thm:main} yields the largest $\delta_Q$ among the sequences of polynomials of the form 
\begin{equation*}
    Q_N(x) = x^{a_N} (1 - x)^{b_N}
\end{equation*}
for some integers $a_N, b_N \geq 0$.

\begin{theorem}\label{thm:optimal}
    Let $\{a_N\}_{N \geq 1}$ and $\{b_N\}_{N \geq 1}$ be sequences of nonnegative integers with the property that, for every positive integer $N$, there exists a polynomial $P \in S_N$ such that $x^{a_N} (1 - x)^{b_N}$ divides $P(x)$ in $\mathbb{Z}[x]$.
    Then $\liminf_{N \,\to\, \infty} (a_N + b_N)/N \leq 5 / 6$.
\end{theorem}

It seems likely that considering more involved ``explicit'' sequences of polynomials $\{Q_N(x)\}$ makes it possible to achieve larger values of $\delta_Q$.
For instance, letting
\begin{equation}\label{equ:QN-2-3-19}
    Q_N(x) := x^{\lfloor N/2\rfloor}(1-x)^{\lfloor N/3\rfloor}(1-x-x^2)^{\lfloor N/19\rfloor} ,
\end{equation}
a computation shows that for every positive integer $N \leq 1000$ there exists a polynomial $P \in S_N$ such that $Q_N(x)$ divides $P(x)$ in $\mathbb{Z}[x]$ (see the code in the Appendix).
If the previous property indeed holds for every positive integer $N$, then
\begin{equation*}
    \delta_Q = \frac{1}{2} + \frac{1}{3} + \frac{2}{19} = 0.938... ,
\end{equation*}
which would improve the value of $\delta_Q$ given by the sequence of Theorem~\ref{thm:main}.

\section{Preliminaries}

This section collects some preliminary results. Except for Lemma ~\ref{lem:criterion}, these are all well-known results, and including/referencing their proofs is just for completeness.

The first lemma provides the evaluation of an integral.

\begin{lemma}\label{lem:binomial-integral}
	Let $a$ and $b$ be nonnegative integers.
	Then
	\begin{equation}\label{equ:binomial-integral}
		\int_0^1 x^a (1 - x)^b \mathrm{d}x = \frac1{(a + b + 1) \binom{a + b}{a}} .
	\end{equation}
\end{lemma}
\begin{proof}
	Let $J(a,b)$ denote the integral on the left-hand side of~\eqref{equ:binomial-integral}.
	Then $J(a, 0) = 1/(a + 1)$.
	If $b \geq 1$ then integrating by parts yields that $J(a, b) = J(a + 1, b - 1) \cdot b / (a + 1)$.
	Thus the claim follows easily by induction.
\end{proof}

The next lemma states a useful divisibility property.

\begin{lemma}\label{lem:binomial-divisibility}
	Let $a,b,N$ be integers such that $a, b \geq 0$ and $a + b < N$.
	Then $(a + b + 1) \binom{a + b}{a}$ divides $\ell_N$.
\end{lemma}
\begin{proof}
	The claim follows from Lemma~\ref{lem:binomial-integral} and from the fact that, as explained in the introduction, if $P_N \in \mathbb{Z}[x]$ is a polynomial such that $\deg(P_N) < N$ then the denominator of $I(P_N)$ divides $\ell_N$.
\end{proof}

Recall the following elementary lemma on solutions to a linear Diophantine equation.

\begin{lemma}\label{lem:bezout}
	Let $c_1, \dots, c_m$ be integers.
	Then the equation $\sum_{n = 1}^m c_n y_n = 1$ has a solution in integers $y_1, \dots, y_m$ if and only if $\gcd\{c_1, \dots, c_m\} = 1$.
	Furthermore, if a solution exists then there exist infinitely many solutions.
\end{lemma}
\begin{proof}
	See, e.g., \cite[Theorem~1.1]{MR3970983}.
\end{proof}

The following lemma is fundamental for the proof of Theorem~\ref{thm:main}.
Hereafter, let $\nu_p$ denote the $p$-adic valuation.

\begin{lemma}\label{lem:criterion}
    Let $a,b,N$ be integers such that $a, b \geq 0$ and $a + b < N$.
    Then the following statements are equivalent.
    \begin{enumerate}[label=(S\arabic*)]
        \item\label{lem:criterion:ite:0} There exists a polynomial $P \in S_N$ such that $x^a (1 - x)^b$ divides $P(x)$ in $\mathbb{Z}[x]$.
        \item\label{lem:criterion:ite:1} There exist infinitely many polynomials $P \in S_N$ such that $x^a (1 - x)^b$ divides $P(x)$ in $\mathbb{Z}[x]$.
        \item\label{lem:criterion:ite:2} For every prime number $p$ not exceeding $N$, there exists an integer $n_p$ such that $0 \leq n_p < N - a - b$ and 
        \begin{equation*}
        	\nu_p\!\left((a + b + n_p + 1) \binom{a + b + n_p}{a}\right) \geq \left\lfloor \frac{\log N}{\log p} \right\rfloor .
        \end{equation*}
    \end{enumerate}
\end{lemma}
\begin{proof}
	For the sake of notation, put $R := N - a - b - 1$.
    Let $P \in \mathbb{Z}[x]$ be a polynomial such that $\deg(P) < N$ and $x^a (1 - x)^b$ divides $P(x)$ in $\mathbb{Z}[x]$.
    This is equivalent to the equality
    \begin{equation*}
        P(x) = x^a (1 - x)^b \sum_{n \,=\, 0}^R y_n (1 - x)^n ,
    \end{equation*}
    for some integers $y_0, \dots, y_R$.
    From Lemma~\ref{lem:binomial-integral} it follows that
    \begin{equation*}
        I(P_N) := \int_0^1 P(x) \mathrm{d}x
        = \sum_{n \,=\, 0}^R y_n \int_0^1  x^a (1 - x)^{b + n} \mathrm{d}x
        = \sum_{n \,=\, 0}^R \frac{y_n}{(a + b + n + 1)\binom{a + b + n}{a}} .
    \end{equation*}
    Thus, setting
    \begin{equation*}
    	c_n := \frac{\ell_N}{(a + b + n + 1)\binom{a + b + n}{a}} 
    \end{equation*}
    for each $n \in\{0,\dots,R\}$, it turns out that $P \in S_N$ if and only if
    \begin{equation}\label{lem:criterion:equ:1}
        \sum_{n \,=\, 0}^R c_n y_n = 1 .
    \end{equation}
    By Lemma~\ref{lem:binomial-divisibility}, the coefficients $c_0, \dots, c_R$ are integers.
    Hence, by Lemma~\ref{lem:bezout}, a necessary and sufficient condition for the existence of integers $y_0, \dots, y_R$ satisfying~\eqref{lem:criterion:equ:1} is that
    \begin{equation}\label{lem:criterion:equ:2}
        \gcd\{c_0, \dots, c_R\} = 1 .
    \end{equation}
    In fact, if~\eqref{lem:criterion:equ:2} holds then by Lemma~\ref{lem:bezout} there are infinitely many integer tuples $(y_0, \dots, y_R)$ that satisfy \eqref{lem:criterion:equ:1}.
    Thus \ref{lem:criterion:ite:0} and \ref{lem:criterion:ite:1} are both equivalent to \eqref{lem:criterion:equ:2}.
    Note that each of the integers $c_0, \dots, c_R$ is a divisor of $\ell_N$ and, consequently, its prime factors do no exceed $N$.
    Hence \eqref{lem:criterion:equ:2} is equivalent to the claim that for every prime number $p$ that does not exceed $N$ there exists and integer $n_p$ such that $0 \leq n_p \leq R$ and
    \begin{equation*}
        \nu_p\!\left((a + b + n_p + 1)\binom{a + b + n_p}{a}\right)
            \geq \nu_p(\ell_N)
            = \left\lfloor \frac{\log N}{\log p} \right\rfloor ,
    \end{equation*}
    which is statement \ref{lem:criterion:ite:2}.
    Therefore, the equivalence of \ref{lem:criterion:ite:0}, \ref{lem:criterion:ite:1}, and \ref{lem:criterion:ite:2} follows.
\end{proof}

Finally, recall the following theorem of Kummer~\cite{MR1578793}, which concerns the $p$-adic valuation of a binomial coefficient.

\begin{theorem}\label{thm:kummer}
	Let $p$ be a prime number and let $a,b$ be nonnegative integers.
	Then the $p$-adic valuation of $\binom{a + b}{a}$ is equal to the number of carries that occur when performing the addition of $a$ and $b$ in base $p$.
\end{theorem}
\begin{proof}
	See, e.g., \cite[Proposition 2.2]{MR4199098}.
\end{proof}

\section{Proof of Theorem~\ref{thm:main}}

Let $N$ be a positive integer and let $p$ be a prime number not exceeding $N$.
For the sake of brevity, put $M := \lfloor N / 6 \rfloor$ and $v := \lfloor \log N /\! \log p \rfloor$.
Since $2M \leq 3M < N < p^{v + 1}$, there exist (unique) integers $a_0,\dots,a_v \in {[0,p)}$ and $b_0, \dots, b_v \in {[0,p)}$ such that
\begin{equation}\label{equ:2M-3M-expansions}
	2M = \sum_{i \,=\, 0}^v a_i p^i \quad\text{ and }\quad
	3M = \sum_{i \,=\, 0}^v b_i p^i .
\end{equation}
Let $i_1$ be the minimal integer such that $0 \leq i_1 \leq v$ and
\begin{equation}\label{equ:i1-condition}
	a_i + b_i \geq p - 1 \text{ for every }  i \in {[i_1, v)} \cap \mathbb{Z} .
\end{equation}
(Note that $i_1$ exists because condition~\eqref{equ:i1-condition} is vacuously true for $i_1 = v$.)
Define
\begin{equation}\label{equ:np-definition}
	n_p := \sum_{i \,=\, 0}^{i_1 - 1} (p - 1 - a_i - b_i) p^i .
\end{equation}
In light of Lemma~\ref{lem:criterion}, it suffices to prove that
\begin{enumerate}[label=(C\arabic*),labelindent=1em,itemsep=0.5em,topsep=0.5em]
	\item\label{ite:cond1} $n_p \geq 0$;
	\item\label{ite:cond2} $p^v$ divides $(5M + n_p + 1) \binom{5M + n_p}{2M}$;
	\item\label{ite:cond3} $5M + n_p + 1 \leq N$.
\end{enumerate}

\begin{proof}[Proof of \ref{ite:cond1}]
	If $i_1 = 0$ then $n_p = 0$ by~\eqref{equ:np-definition}, and so \ref{ite:cond1} is true.
	Suppose that $i_1 > 0$.
	The minimality of $i_1$ with regard to \eqref{equ:i1-condition} implies that 
	\begin{equation}\label{equ:ai1-1bi1-1}
		a_{i_1 - 1} + b_{i_1 - 1} < p - 1
	\end{equation}
	From \eqref{equ:np-definition} and \eqref{equ:ai1-1bi1-1} it follows that
	\begin{equation*}
		n_p \geq p^{i_1-1} + \sum_{i \,=\, 0}^{i_1 - 2} (p - 1 - a_i - b_i) p^i
			\geq p^{i_1 - 1} - \sum_{i \,=\, 0}^{i_1 - 2} (p - 1) p^i
			= p^{i_1 - 1} - (p^{i_1 - 1} - 1) 
			= 1 ,
	\end{equation*}
	which implies \ref{ite:cond1}.
\end{proof}

\begin{proof}[Proof of \ref{ite:cond2}]
	Let $i_2$ be the maximal integer such that $i_1 \leq i_2 \leq v$ and
	\begin{equation}\label{equ:i2-condition}
		a_i + b_i = p - 1 \quad\text{ for every }  i \in {[i_1, i_2)} \cap \mathbb{Z} .
	\end{equation}
	(Note that $i_2$ exists because condition~\eqref{equ:i2-condition} is vacuously true for $i_2 = i_1$.)
	From \eqref{equ:2M-3M-expansions} and \eqref{equ:np-definition} it follows that
	\begin{equation}\label{equ:2M-three-parts}
		2M = \sum_{i \,=\, 0}^{i_1 - 1} a_i p^i + \sum_{i \,=\, i_1}^{i_2 - 1} a_i p^i + \sum_{i \,=\, i_2}^v a_i p^i
	\end{equation}
	and
	\begin{equation}\label{equ:3M-three-parts}
		3M + n_p = \sum_{i \,=\, 0}^{i_1 - 1} (p - 1 - a_i) p^i + \sum_{i \,=\, i_1}^{i_2 - 1} b_i p^i + \sum_{i \,=\, i_2}^v b_i p^i .
	\end{equation}
	Furthermore, inequality \eqref{equ:i1-condition} and the maximality of $i_2$ with regard to \eqref{equ:i2-condition} imply that
	\begin{equation}\label{equ:ai2bi2}
		i_2 = v, \text{ or } i_2 < v \text{ and } a_{i_2} + b_{i_2} \geq p .
	\end{equation}
	From \eqref{equ:i1-condition}, \eqref{equ:i2-condition}, \eqref{equ:2M-three-parts}, \eqref{equ:3M-three-parts}, and \eqref{equ:ai2bi2} it follows that at least $v - i_2$ carries occur when adding $2M$ and $3M + n_p$ in base $p$.
	Thus Theorem~\ref{thm:kummer} implies that
	\begin{equation}\label{equ:p-adic-second-factor}
		\nu_p\!\left( \binom{5M + n_p}{2M} \right) \geq v - i_2 .
	\end{equation}
	From \eqref{equ:i2-condition}, \eqref{equ:2M-three-parts}, and \eqref{equ:3M-three-parts} it follows that
	\begin{equation}\label{equ:5Mnp1-pi2}
		5M + n_p + 1 = 1 + \sum_{i \,=\, 0}^{i_2 - 1} (p - 1) p^i + \sum_{i \,=\, i_2}^v (a_i + b_i) p^i
			= p^{i_2} + \sum_{i \,=\, i_2}^v (a_i + b_i) p^i ,
	\end{equation}
	and consequently
	\begin{equation}\label{equ:p-adic-first-factor}
		\nu_p(5M + n_p + 1) \geq i_2 .
	\end{equation}
	Combining \eqref{equ:p-adic-second-factor} and \eqref{equ:p-adic-first-factor} yields \ref{ite:cond2}.
\end{proof}

\begin{proof}[Proof of \ref{ite:cond3}]
	If $p^{i_2} \leq M$ then \eqref{equ:5Mnp1-pi2} and \eqref{equ:2M-3M-expansions} imply that
	\begin{equation*}
		5M + n_p + 1 \leq M + \sum_{i \,=\, i_2}^v (a_i + b_i) p^i
			\leq M + 2M + 3M 
			= 6M \leq N ,
	\end{equation*}
	which is \ref{ite:cond3}.
	Suppose that $p^{i_2} > M$.
	Then \eqref{equ:2M-3M-expansions} implies that
	\begin{equation}\label{equ:other-2M-3M-bounds}
		2p^{i_2} > 2M \geq \sum_{i \,=\, i_2}^v a_i p^i \quad\text{ and }\quad
		3p^{i_2} > 3M \geq \sum_{i \,=\, i_2}^v b_i p^i .
	\end{equation}
	In turn, from \eqref{equ:other-2M-3M-bounds} it follows that 
	\begin{align}
		&a_{i_2} \leq 1 \text{ and } b_{i_2} \leq 2 ,\label{equ:bounds-ai2-bi2} \\
		&a_i = 0 \text{ for every } i \in [i_2+1, v] \cap \mathbb{Z} ,\label{equ:bounds-ai2+} \\
		&b_j = 0 \text{ for every } j \in [i_2+2, v] \cap \mathbb{Z} . \label{equ:bounds-bi2+}
	\end{align}
    For the sake of contradiction, suppose that $i_2 < v$ and $b_{i_2 + 1} \geq 1$.
    Then \eqref{equ:other-2M-3M-bounds} implies that 
    \begin{equation*}
        3p^{i_2} > b_{i_2} p^{i_2} + p^{i_2 + 1} ,
    \end{equation*}
    which in turn gives that $p = 2$ and $b_{i_2} = 0$.
    But, since $i_2 < v$, from \eqref{equ:ai2bi2} and \eqref{equ:bounds-ai2-bi2} it follows that
    \begin{equation*}
        2 = p \leq a_{i_2} + b_{i_2} \leq 1 + 0 = 1,
    \end{equation*}
    which is impossible.
    Thus, if $i_2 < v$ then $b_{i_2 + 1} = 0$.
	Therefore, regardless of whether $i_2 < v$ or not, from \eqref{equ:5Mnp1-pi2}, \eqref{equ:bounds-ai2+}, and \eqref{equ:bounds-bi2+}  it follows that
	\begin{equation}\label{equ:5mnp1-is-ai2-bi2}
		5M + n_p + 1 = (a_{i_2} + b_{i_2} + 1) p^{i_2} .
	\end{equation}
	At this point, it only remains to patiently check four cases.
	
	If $b_{i_2} = 2$ then $2p^{i_2} \leq 3M$ by \eqref{equ:2M-3M-expansions}.
	This and \eqref{equ:bounds-ai2-bi2} imply that
	\begin{equation*}
		(a_{i_2} + b_{i_2} + 1) p^{i_2} \leq (1 + 2 + 1) p^{i_2} = 4 p^{i_2} \leq 6M \leq N ,
	\end{equation*}
	which combined with~\eqref{equ:5mnp1-is-ai2-bi2} yields \ref{ite:cond3}.
	The other cases proceed similarly (further references to \eqref{equ:2M-3M-expansions}, \eqref{equ:bounds-ai2-bi2}, and \eqref{equ:5mnp1-is-ai2-bi2} are omitted).
	
	If $b_{i_2} \leq 1$ and $a_{i_2} = 1$ then $p^{i_2} \leq 2M$ and so
	\begin{equation*}
		(a_{i_2} + b_{i_2} + 1) p^{i_2} \leq (1 + 1 + 1) p^{i_2} = 3 p^{i_2} \leq 6M \leq N .
	\end{equation*}
	If $b_{i_2} = 1$ and $a_{i_2} = 0$ then $p^{i_2} \leq 3M$ and so
	\begin{equation*}
		(a_{i_2} + b_{i_2} + 1) p^{i_2} \leq (0 + 1 + 1) p^{i_2} = 2 p^{i_2} \leq 6M \leq N .
	\end{equation*}
	If $b_{i_2} = 0$ and $a_{i_2} = 0$ then 
	\begin{equation*}
		(a_{i_2} + b_{i_2} + 1) p^{i_2} \leq (0 + 0 + 1) p^{i_2} = p^{i_2} \leq p^v \leq N .
	\end{equation*}	
	This proves \ref{ite:cond3}.
\end{proof}

The proof of Theorem~\ref{thm:main} is complete.

\section{Proof of Theorem~\ref{thm:optimal}}

Let $\varepsilon$ be an arbitrary positive real number.
It suffices to prove that 
\begin{equation}\label{equ:optimal-proof:1}
	a_N + b_N \leq \big( \tfrac{5}{6} + \varepsilon \big) N 
\end{equation}
for infinitely many positive integers $N$.
Let $q$ be a sufficiently large (depending on $\varepsilon$) prime number, put $N := 2(q - 1)$, and let $r$ be the minimal prime number greater than $N/3$.
A result of analytic number theory (see, e.g., the paper by Baker, Harman, and Pintz~\cite{MR1851081}) states that there exists an absolute constant $\theta \in (0, 1)$ such that for every sufficiently large real number $x$ the interval $\left[x, x + x^\theta\right]$ contains a prime number.
Hence 
\begin{equation}\label{equ:optimal-proof:2}
	r < \big(\tfrac{1}{3} + \varepsilon \big) N .
\end{equation}
By the hypotheses, there exists a polynomial $P \in S_N$ such that $x^{a_N}(1 - x)^{b_N}$ divides $P(x)$ in $\mathbb{Z}[x]$.
Note that $\int_0^1 P(x) \mathrm{d} x = \int_0^1 P(1 - x) \mathrm{d} x$.
Hence, without loss of generality, it is possible to assume that $a_N \leq b_N$.
Since $q < N < 2q < q^2$ and $2r < N < 3r < r^2$, from Lemma~\ref{lem:criterion} it follows that for each prime $p \in \{q, r\}$ there exists an integer $n_p$ such that $0 \leq n_p < N - a_N - b_N$ and
\begin{equation}\label{equ:optimal-proof:3}
	\nu_p\!\left((a_N + b_N + n_p + 1) \binom{a_N + b_N + n_p}{a_N}\right) \geq 1 .
\end{equation}
If $q$ divides $a_N + b_N + n_q + 1$ then, since
\begin{equation}\label{equ:optimal-proof:3-1}
	a_N + b_N + n_q + 1 \leq N < 2q ,
\end{equation}
it follows that $a_N + b_N + n_q + 1 = q$.
Consequently
\begin{equation*}
	a_N + b_N \leq q - 1 = \tfrac1{2} N ,
\end{equation*}
which implies \eqref{equ:optimal-proof:1}, as desired.

Similarly, if $r$ divides $a_N + b_N + n_r + 1$ then, since
\begin{equation}\label{equ:optimal-proof:3-2}
	a_N + b_N + n_r + 1 \leq N < 3r ,
\end{equation}
it follows that $a_N + b_N + n_r + 1 \in \{r, 2r\}$.
This and \eqref{equ:optimal-proof:2} imply that
\begin{equation*}
	a_N + b_N < 2r < \big(\tfrac{2}{3} + 2\varepsilon \big)N ,
\end{equation*}
which in turn implies \eqref{equ:optimal-proof:1}, as desired.

Suppose that $q$ does not divide $a_N + b_N + n_q + 1$ and that $r$ does not divide $a_N + b_N + n_r + 1$.
From \eqref{equ:optimal-proof:3} it follows that $q$ divides
\begin{equation*}
	\frac{(a_N + b_N + n_q + 1)!}{a_N! (b_N + n_q)!} .
\end{equation*}
Consequently, again by \eqref{equ:optimal-proof:3-1}, 
\begin{equation}\label{equ:optimal-proof:4}
	b_N \leq b_N + n_q \leq q - 1 = \tfrac1{2} N .
\end{equation}
Similarly, from \eqref{equ:optimal-proof:3} it follows that $r$ divides
\begin{equation*}
	\frac{(a_N + b_N + n_r + 1)!}{a_N! (b_N + n_r)!} .
\end{equation*}
Consequently, again by \eqref{equ:optimal-proof:3-2}, either: $a_N < r$ and $b_N + n_r < 2r$; or $a_N < 2r$ and $b_N + n_r < r$.
Recalling the inequalities $a_N \leq b_N$ and \eqref{equ:optimal-proof:2}, it follows that
\begin{equation}\label{equ:optimal-proof:5}
	a_N < r < \big(\tfrac1{3} + \varepsilon \big) N .
\end{equation}
Combining \eqref{equ:optimal-proof:4} and \eqref{equ:optimal-proof:5} yields \eqref{equ:optimal-proof:1}, as desired.

The proof is complete.

\providecommand{\bysame}{\leavevmode\hbox to3em{\hrulefill}\thinspace}
\providecommand{\MR}{\relax\ifhmode\unskip\space\fi MR }
\providecommand{\MRhref}[2]{%
	\href{http://www.ams.org/mathscinet-getitem?mr=#1}{#2}
}
\providecommand{\href}[2]{#2}

\addresseshere

\newpage
\appendix
\section*{Appendix}

The following \textsf{Mathematica} code checks that, letting $Q_N(x)$ be given by \eqref{equ:QN-2-3-19}, for every positive integer $N \leq 1000$ there exists a polynomial $P \in S_N$ such that $Q_N(x)$ divides $P(x)$ in $\mathbb{Z}[x]$.

\begin{lstlisting}
(* Initialize the list of possible exceptions. *)
ris = {};

Monitor[
 Do[
  (* Compute d_n = lcm(1,...,n). *)
  dn = LCM @@ Range[n];

  (* Define the degree of the polynomial q. *)
  m = n - IntegerPart[n/2] - IntegerPart[n/3] 
      - 2*IntegerPart[n/19] - 1;

  (* Coefficients a_0,...,a_m. *)
  coeff = Array[a, m + 1, 0];

  (* Polynomial q(x). *)
  q = Sum[coeff[[i + 1]]*x^i, {i, 0, m}];

  (* Find an integer solution. *)
  sol = FindInstance[
    Integrate[
      x^IntegerPart[n/2] *
      (1 - x)^IntegerPart[n/3] *
      (1 - x - x^2)^IntegerPart[n/19] * q,
      {x, 0, 1}
    ] == 1/dn,
    coeff, Integers
  ];

  (* If no solution is found, save n. *)
  If[sol == {}, AppendTo[ris, {n}]],

 {n, 3, 1000}
 ],
 n
];

Print[" "];

If[ris == {},
  Print["No exception."],
  (
   Print["Exceptions:"];
   Print[ris]
  )
];
\end{lstlisting}

\end{document}